\documentclass{article}
\usepackage{amsmath}
\usepackage{amsthm}
\usepackage{tikz}
\usepackage{amssymb}
\usepackage{mathtools}
\usepackage{mathabx}
\usepackage{algorithm}
\usepackage{algpseudocode}
\usepackage{hyperref}
\usepackage[toc]{appendix}
\usepackage[nameinlink]{cleveref}

\DeclarePairedDelimiter\floor{\lfloor}{\rfloor}
\DeclareMathAlphabet{\mathbbold}{U}{bbold}{m}{n}

\usepackage[numbers,sort&compress]{natbib}

\def\parens#1{\left(#1\right)}

\usepackage{tikz}
\usetikzlibrary{angles,quotes}

\usepackage{thmtools}
\usepackage{thm-restate}

\newtheorem{theorem}{Theorem}
\newtheorem{lemma}{Lemma}
\newtheorem*{lemma*}{Lemma}
\newtheorem{corollary}{Corollary}

\newtheorem*{claim*}{Claim}

\theoremstyle{remark}

\newtheorem*{remark*}{Remark}

\newtheorem*{example*}{Example}

\title{\textbf{The Prisoners and the Swap: \\ Less than Half is Enough}}
\author{Uri Mendlovic
\\  \href{mailto:urimend@gmail.com}{urimend@gmail.com}}
\date{June 2024}

\begin{document}

\maketitle

\begin{abstract}
We improve the solution of the classical prisoners and drawers riddle, where all prisoners can find their number using the pointer-following strategy, provided that the prisoners can send a spy to inspect all drawers and swap one pair of numbers. In the traditional approach, each prisoner may need to open up to half of the drawers. We show that this strategy is sub-optimal. Remarkably, a single swap allows all $n$ prisoners to find their number by opening only $\frac{n \ln \ln n}{\ln n} (1 + o(1))$ drawers in the worst case. We show that no strategy can do better than that by a factor larger than two. Efficiently constructing such a strategy is harder, but we provide an explicit efficient strategy that requires opening only $O(\frac{n \log \log n}{\log n})$ drawers by each prisoner in the worst case.
\end{abstract}

\section{Introduction}
\subsection{The prisoners and the drawers}
The famous prisoners and drawers riddle originated in a paper by  Anna Gál and Peter Bro Miltersen \cite{succint_2003}. A simpler version of the problem given in \cite{math_puzzles}. Here we use the phrasing given in \cite{analytic_combinatorics}:

\begin{quotation}
The director of a prison offers 100 death row prisoners, who are numbered from 1 to 100, a last chance. A room contains a cupboard with 100 drawers. The director randomly puts one prisoner's number in each closed drawer. The prisoners enter the room, one after another. Each prisoner may open and look into 50 drawers in any order. The drawers are closed again afterwards. If, during this search, every prisoner finds their number in one of the drawers, all prisoners are pardoned. If even one prisoner does not find their number, all prisoners die. Before the first prisoner enters the room, the prisoners may discuss strategy — but may not communicate once the first prisoner enters to look in the drawers. What is the prisoners' best strategy?
\end{quotation}

\cite{succint_2003} credited Sven Skyum for solving this problem, though they left the solution as a riddle to the reader. The solution is based on the well known pointer-following strategy. The drawers are numbered and each prisoner first opens the drawer labeled with their own number. If the drawer contains the number of another prisoner, they next open the drawer labeled with this number, and so on. Each prisoner follows this strategy until they find their number.

In the above phrasing the numbers are put into the drawers in a random permutation. The prisoners all succeed if and only if the permutation contains no cycle larger than 50. This happens with a probability of about 31\%. This method is optimal, as proven in \cite{locker_2006}.

\subsection{The spy and the swap}
Another version of the problem allows the prisoners to send a spy after the numbers are placed into the drawers but before the prisoners enter the room. The spy can inspect all drawers and swap a single pair of numbers. The spy cannot communicate with the prisoners after entering the room.

In this version all prisoners can find their number by opening 50 drawers, regardless of the assignment of the original numbers. In other words, all prisoners escape with probability 1. This is achieved by instructing the spy to make a single swap that splits the largest cycle into two, ensuring no cycles larger than 50 remain.

This paper considers the problem of finding a strategy for the spy and the prisoners such that each prisoner finds their number by opening fewer than half the drawers, irrespective of the initial number placement. Surprisingly, despite the abundance of papers discussing this problem (and a Veritasium video with millions of views \cite{veritasium}), we believe this problem was never studied before. 

As far as we know, the only work to consider a swap and opening less than half of the drawers was carried out by Czumaj et al \cite{haystack}, but their work is restricted to the case where each prisoner can open just two drawers. This of course is not enough to allow all prisoners to find their numbers. Instead, they assume a uniformly random permutation and count the number of successful prisoners on average, or equivalently, the probability that a uniformly chosen prisoner is successful. Our approach can be viewed as a generalization of their method, allowing more than two drawers to be opened.

\subsection{Our contribution}
Our main contribution is the following result.

\begin{theorem}
\label{theorem:probabilistic}
The spy and the $n$ prisoners have a strategy such that for any initial permutation the spy makes a single swap and then all prisoners find their number by opening $\frac{n \ln \ln n}{\ln n} (1 + o(1))$ drawers, and this amount is optimal up to a factor of two in the number of drawers opened.
\end{theorem}

We use $o(1)$ to denote a function that diminishes to zero as $n$ grows to infinity.

We prove this theorem in \Cref{sec:probabilistic} using a probabilistic argument.

Even though \Cref{theorem:probabilistic} shows the existence of a deterministic method that works for all initial permutations, it does not show how to construct such a solution efficiently, with time complexity polynomial in $n$. This is the scope of the next theorem:

\begin{restatable}{theorem}{deterministic}
\label{theorem:deterministic}
The spy and the $n$ prisoners have an efficient strategy such that for any initial permutation the spy makes a single swap and then all prisoners find their number by opening $O(\frac{n \log \log n}{\log n})$ drawers. The time complexity of the strategy is $O(n^2)$ for each participant.
\end{restatable}

Constructing such a strategy is much harder, requiring the development of several new tools. This work is described in \Cref{sec:deterministic}, which spans most of this paper.

\section{Existence and optimality}
\label{sec:probabilistic}
In this section we prove \Cref{theorem:probabilistic} by providing a lower bound and an upper bound on the minimal number of drawers needed for allowing all prisoners to find their number regardless of the original permutation. Recall that the strategy includes sending a spy to inspect all drawers and swap two numbers.

We use the following well-known result:
\begin{lemma}
\label{lemma:dickman}
As $n$ grows to infinity, if $u = O(n ^ {1/2 - \epsilon})$ for some $\epsilon > 0$, the probability that a uniformly random permutation of $n$ elements does not have a cycle larger than $\frac{n}{u}$ is
$$\rho(u) \cdot (1 + o(1)) = u ^ {-u (1 + o(1))}$$
where $\rho(u)$ is the Dickman function.
\end{lemma}
\begin{proof}
The estimate $\rho(u) \cdot (1 + o(1))$ follows from Theorem 3 in \cite{random_permutation}. Note that the condition there is satisfied by the fact that $u = O(n ^ {1/2 - \epsilon})$ for some $\epsilon > 0$.

The estimate $u ^ {-u (1 + o(1))}$ is proven, for example, in \cite{integers}. Note that the $o(1)$ in the exponent here dominates the multiplicative $o(1)$ in $\rho(u) \cdot (1 + o(1))$.

\end{proof}

\subsection{A lower bound}
To lower bound the number of drawers opened in the worst case we consider the following variation of the problem: after inspecting the drawers, the spy cannot make a swap or any other change. Instead the spy communicates one of $m$ possible messages to all prisoners.

We prove the following Lemma.
\begin{lemma}
\label{lemma:communication_lower_bound}
If the spy is allowed to communicate one of $m$ messages, the $n$ prisoners must open at least
\begin{equation}
\label{eq:communication_lower_bound}
\frac{n \ln \ln m}{\ln m} (1 + o(1))
\end{equation}
drawers to find their number in the worst case.
\end{lemma}
\begin{proof}
Let $k$ be the number of drawers opened in the worst case. We fix a specific message and inspect the strategy used by the prisoners after receiving this message from the spy. \cite{locker_2006} showed that such a strategy, if applied to all possible assignments of numbers to drawers, works only for a portion of them that is upper bounded by the probability that a random permutation has no cycles larger than $k$, which is given by \Cref{lemma:dickman}. In other words, to cover all assignments the number of communicated values must satisfy:
$$m \ge u ^ {u (1 + o(1))}$$
where $u = n / k$. Taking $\ln(\cdot)$ we get
$$\ln m \ge (1 + o(1)) u \ln u$$
which implies:
$$k = n / u \ge \frac{n \ln \ln m}{\ln m} (1 + o(1))$$
as required.
\end{proof}

We now use the communication scenario as a lower bound for a swap (and no communication).

\begin{corollary}
\label{cor:lower_bound}
If the spy is allowed to make a single swap, the $n$ prisoners must open at least $\frac{n \ln \ln n}{2 \ln n} (1 + o(1))$ drawers to find their number in the worst case.
\end{corollary}
\begin{proof}
Assume, for the sake of contradiction, that a better strategy is possible. Use this strategy for the communication scenario of \Cref{lemma:communication_lower_bound} with $m = \binom{n}{2}$. The spy indicates which pair of drawers should be swapped by sending their indices, and the prisoners can then follow their strategy imagining that the corresponding drawers were swapped.

Substituting $m = \binom{n}{2} = \frac{n^2}{2} (1 + o(1))$ into equation \Cref{eq:communication_lower_bound} gives the required bound.
\end{proof}

Note that \Cref{cor:lower_bound} proves the optimality claim of \Cref{theorem:probabilistic}.

\subsection{An upper bound}
We next prove the existence of a strategy for \Cref{theorem:probabilistic}. The strategy will have the following form: All prisoners open the drawers labeled $1$ to $r$, where $r$ is chosen in advance. The numbers in these $r$ drawers (in their specific order) are used to choose a permutation $\pi$ of all $n$. In other words, the strategy consists of a book of $\frac{n!}{(n-r)!}$ permutations, one for each sequence of values seen in the first $r$ drawers. Then the prisoners use the pointer-following strategy with the drawer labels permuted according to $\pi$.

The spy can influence the chosen permutation by swapping a pair of drawers that includes at least one of the first $r$ drawers. The spy can inspect the largest cycle of the permutation obtained by each of these possible swaps, and choose the one that minimizes the size of the largest cycle.

We now prove the following lemma:
\begin{lemma}
\label{lemma:random_strategy}
For $r = \left\lfloor \frac{n}{\ln n} \right\rfloor$, the $\frac{n!}{(n-r)!}$ permutations in the strategy book can be chosen such that the spy can use a single swap to obtain a permutation with no cycles larger than $\frac{n \ln \ln n}{\ln n} (1 + o(1))$.
\end{lemma}

We note that the existence part of \Cref{theorem:probabilistic} follows immediately from \Cref{lemma:random_strategy} since the total number of drawers opened in the worst case is $r$ plus the size of the largest cycle, which is:
$$r + \frac{n \ln \ln n}{\ln n} (1 + o(1)) = \left\lfloor \frac{n}{\ln n} \right\rfloor +  \frac{n \ln \ln n}{\ln n} (1 + o(1)) = \frac{n \ln \ln n}{\ln n} (1 + o(1))$$

We now prove \Cref{lemma:random_strategy}.
\begin{proof}[Proof of \Cref{lemma:random_strategy}]
Suppose we require all cycles to be not larger than a certain bound $k$ we will choose later. Denote by $p_n$ the probability that a random permutation of $n$ elements satisfies this condition.

Suppose all permutations in the strategy book were chosen uniformly at random from all possible permutations. The probability that composing a given permutation (given by the initial assignment) with a certain permutation in the book yields a permutation with no cycles larger than $k$ is $p_n$, since composing with a uniformly random permutation yields a uniformly random permutation.

We note that there are $r(n-r) + \binom{r}{2}$ reachable values for the first $r$ drawers by making a single swap. Since the lemma requires $r = \left\lfloor \frac{n}{\ln n} \right\rfloor$, the number of reachable values is $rn (1 - o(1))$. The probability of each of these options to succeed is independent, so the overall probability that a certain initial assignment does not have a good swap is:
$$\parens{1 - p_n} ^ {rn (1 - o(1))}$$

By linearity of expectation, the expected number of initial assignments not having a good swap is:
$$n! \cdot \parens{1 - p_n} ^ {rn (1 - o(1))}$$

We require the expected number of initial permutations without a good swap to be less than one, ensuring a positive probability of meeting the requirement for all initial assignments, thereby proving the existence of the required strategy.

We apply $\ln(\cdot)$ on both sides of the inequality
$$n! \cdot \parens{1 - p_n} ^ {rn (1 - o(1))} < 1$$
and use the facts that $n! \le n^n$ and $\ln(1 - x) \le -x$ to obtain the sufficient condition:
$$p_n r (1 - o(1)) > \ln n$$

Recalling that the lemma set $r = \floor{\frac{n}{\ln n}}$, the requirement is that:
\begin{equation}
\label{eq:prob_required}
p_n > \frac{(\ln n)^2}{n} (1 + o(1))
\end{equation}
for some $o(1)$ function.

Our goal is to upper bound the cycle size by $k = n / u = \frac{n \ln \ln n}{\ln n} (1 + f_n)$ where $f_n$ is an $o(1)$ function, as required by the statement of \Cref{lemma:random_strategy}, and still meet condition \Cref{eq:prob_required} to ensure that the permutations in the strategy book can be chosen to meet the requirement.

We use \Cref{lemma:dickman} to bound $p_n$:
$$p_n = u ^ {-u (1 + o(1))} > (\ln n) ^ {-u (1 + o(1))} = n ^ {-\frac{1 + o(1)}{1 + f_n}}$$

The inequality is due to the fact that $u = \frac{\ln n}{\ln \ln n} (1 + o(1))$ is smaller than $\ln n$ for large enough $n$.

Substituting \Cref{eq:prob_required}, the requirement is:
$$n ^ {-\frac{1 + o(1)}{1 + f_n}} \ge \frac{(\ln n)^2}{n} (1 + o(1))$$
taking $\ln(\cdot)$ we get
$$-\frac{(1 + o(1)) \ln n}{1 + f_n} \ge 2 \ln \ln n - \ln n + \ln (1 + o(1)) = 2 \ln \ln n - \ln n + o(1)$$
which simplifies into
$$1 + f_n \ge \frac{(1 + o(1)) \ln n}{\ln n - 2 \ln \ln n - o(1)}$$
to finally obtain
$$f_n \ge \frac{o(1) \ln n + 2 \ln \ln n + o(1)}{\ln n - 2 \ln \ln n - o(1)}.$$
The right hand side is $o(1)$ function so we can choose $f_n = o(1)$ that satisfies this inequality, proving the existence of a strategy with cycles not larger than $\frac{n \ln \ln n}{\ln n} (1 + f_n)$ with $f_n = o(1)$, as promised.

\end{proof}

This concludes the proof of \Cref{theorem:probabilistic}. We note that the strategy given in this section is not efficient. It consists of a book of size
$$\frac{n!}{(n-r)!} \approx n ^ r \approx e ^ n$$

Besides the fact that just describing the strategy requires an exponential amount of space, it is not clear how to construct this strategy deterministically even in exponential time. The next section explicitly provides an efficient strategy.

\section{Efficient solution}
\label{sec:deterministic}
This section provides an explicit strategy that is efficient, in the sense that its time complexity is polynomial in the input length. In \Cref{sec:scheme} we describe the general idea of the efficient solution, while in \Cref{sec:encoding} and \Cref{sec:breaking} we explain two major components in the solution. We then conclude by combining the pieces together and analyzing the performance of the solution. We note that the number of drawers required by this solution is slightly larger than the randomly-constructed strategy given in \Cref{sec:probabilistic}.

\subsection{The scheme}
\label{sec:scheme}
As in the randomly-constructed strategy, all prisoners always open the first $r$ drawers. Prisoners that find their numbers in these drawers are successful and leave the room. The rest of the prisoners all know the numbers in the first $r$ drawers. They renumber these values $1, ..., r$ such that $1$ corresponds to the minimal value in these drawers and so on up to $r$ that corresponds to the maximal value. This way they recover a permutation of $r$ elements: $\pi_0 \in S_r$. Similarly, they renumber all numbers missing from the first $r$ drawers $1, ..., (n-r)$. They use the opened $r$ drawers to obtain a permutation $\pi_1$ of $n-r$ numbers:
$$\pi_1 = f(\pi_0) \in S_{n-r}$$
where $f$ is a function defined by the strategy. They then use the classical pointer-following strategy on the $n-r$ remaining drawers, where drawers are renumbered $1, ..., (n-r)$, and so are the numbers in the drawers. The permutation $\pi_1$ is used to permute the drawer labels.

More explicitly, let $T \subset \{1, ..., n\}$ be the set of $n-r$ numbers missing from the first $r$ drawers. Let $h_T: T \to \{1, ..., n-r\}$ be a function that for each number in $T$ returns its position among the numbers in $T$. Then, if the $i$-th prisoner failed to find their number in the first $r$ drawers, they open box number $r + \pi_1(h_T(i))$ to obtain a number $i' \in T$. If $i' = i$, the prisoner is successful and leaves the room. Otherwise, the prisoner moves on to open the box that prisoner $i'$ would have opened, that is $r + \pi_1(h_T(i'))$.

The core of the strategy is the construction of a function $f$ that the spy can use to eliminate large cycles in the pointer-following phase. We break the function $f: S_r \to S_{n-r}$ into two stages:
$$f = f_1 \circ f_0$$
$$f_0: S_r \to [m]$$
$$f_1: [m] \to S_{n-r}$$
where $m$ is a parameter to be chosen later and $[m] = \{1, ..., m\}$.

The first function $f_0$ maps $\pi_0$ into an index out of $m$ options. In \Cref{sec:encoding} we construct $f_0$ such that for any initial permutation $\pi_0^\textrm{init}$ and a target index $m_0 \in [m]$, there exists a transposition of $\pi_0^\textrm{init}$ such that $f_0$ gives the desired index. Formally:
$$\forall \pi_0^\textrm{init} \in S_r: \forall m_0 \in [m]: \exists (i,j) \in [r]^2: f_0(T_{i \leftrightarrow j}(\pi_0^\textrm{init})) = m_0$$
where $T_{i \leftrightarrow j}$ is the transposition of $i$ and $j$ (that is, a permutation that swaps $i$ and $j$ and leaves all other values unchanged). This property allows the spy to choose the output of $f_0$ by swapping the content of a pair of drawers out of the first $r$ drawers.

The second function $f_1$ maps the index to a permutation in a set $F$ of up to $m$ permutations of $n-r$ elements. In \Cref{sec:breaking} we construct a set $F \subset S_{n-r}$ such that any permutation of $n-r$ elements $\pi_1^\textrm{init}$ can be composed with a permutation in the set to obtain a permutation with no cycle larger than a certain bound $k$. Formally:
$$\forall \pi_1^\textrm{init} \in S_{n-r}: \exists \pi_1' \in F: \mathcal{L}_{\text{max}} (\pi_1^\textrm{init} \circ \pi_1') \le k$$
where $\mathcal{L}_{\text{max}}(\pi)$ denotes the length of the longest cycle in a permutation $\pi$.

The functions $f_0$ and $f_1$, having the mentioned properties, allow the spy to choose a swap such that all prisoners find their number by opening no more than $r + k$ drawers: the first $r$ drawers and up to $k$ of the remaining drawers. We should emphasize that the swap is made between the first $r$ drawers, so the set of numbers in the first $r$ drawers is not changed by the swap, as does the permutation $\pi_1$ describing the remaining $n-r$ drawers after renumbering.

Clearly, the construction cannot work for all possible values of $r, m, k$. The constructions in \Cref{sec:encoding} and \Cref{sec:breaking} impose conditions on these numbers.

\subsection{Encoding a message using a swap}
\label{sec:encoding}
We now construct the function $f_0: S_r \to [m]$ such that the output of $f_0$ can be fully controlled using a single swap on the input permutation.

We break $f_0$ into two stages, first transforming the permutation to a vector of $d = \floor{\frac{r}{3}}$ bits, and then transforming the vector into one of $m$ numbers, where $m$ is at least $\frac{d^2}{16}$. Formally:
$$f_0 = g_1 \circ g_0$$
$$g_0: S_r \to \{0, 1\}^d$$
$$g_1: \{0, 1\}^d \to [m]$$

\begin{theorem}
\label{theorem:swap2flips}
For any positive integer $r$, there is a function $g_0: S_r \to \{0, 1\}^d$ with $d = \floor{\frac{r}{3}}$ such that flipping any two bits in the output of $g_0$ is possible by making a single swap on the input permutation. Formally\footnote{We use $\oplus$ to denote the bitwise XOR operation on binary vectors.}:
\begin{multline*}
\forall \pi_0^\mathrm{init} \in S_r: \forall (i_0,i_1) \in [d]^2: \exists (j_0,j_1) \in [r]^2:\\
g_0(T_{j_0 \leftrightarrow j_1}(\pi_0^\mathrm{init})) = g_0(\pi_0^\mathrm{init}) \oplus e_{i_0} \oplus e_{i_1}
\end{multline*}
where $e_i \in \{0, 1\}^d$ is the vector that has $1$ only in its $i$-th coordinate.
\end{theorem}
\begin{proof}
The function $g_0$ is defined by dividing the input permutation $\pi_0$ into $d = \floor{\frac{r}{3}}$ triples of numbers, the $i$-th triple being:
$$\left( \pi_0(3i - 2), \pi_0(3i - 1), \pi_0(3i) \right)$$
Then each such triple is transformed into a permutation of three elements by renumbering the three values in their order. This permutation is mapped to a single bit by taking the parity of the permutation. For example, $(10, 11, 17)$ is mapped to $(1, 2, 3)$ which corresponds to the identity permutation, with parity $0$.

This function has the desired property: In order to flip output bits $i_0$ and $i_1$ a single swap is made between the $i_0$-th triple and the $i_1$-th triple. There always exists such a swap that flips the parity of both triples. We note that the ordering of the six values in the two triples determines the parities and the chosen swap. To prove that a swap always exists, one may enumerate all cases using a computer program or manually check the few cases remaining after identifying additional symmetries of the problem.

To illustrate we give the following example. Suppose the two triples are:
$$(80, 90, 48) \quad (17, 62, 39)$$
The parities are $0$ and $1$ correspondingly. Swapping $80$ and $39$ we obtain the triples:
$$(39, 90, 48) \quad (17, 62, 80)$$
whose parities are $1$ and $0$. Both parities have been flipped as required.

\end{proof}

\begin{theorem}
\label{theorem:flips2message}
For any positive integer $d$ there is a function $g_1: \{0, 1\}^d \to [m]$ with $m > \parens{\frac{d}{4}}^2$ such that the output of $g_1$ can be fixed arbitrarily by flipping two bits in the input vector. Formally:
$$\forall v \in \{0, 1\}^d: \forall m_0 \in [m]: \exists (i_0,i_1) \in [d]^2:
g_1(v \oplus e_{i_0} \oplus e_{i_1}) = m_0$$
\end{theorem}
\begin{proof}
Given a vector $v$ of $2^a$ bits, the bits can be numbered $0, ..., (2^a-1)$. Then one can compute the bitwise XOR of all indices of $1$ bits (that is, $\bigoplus_{i: v_i=1} i$). Flipping the $i$-th bit in the vector xors $i$ to the result, so any target output is achievable by flipping one bit in the vector. We note that this is similar to the Hamming parity-check matrix, computing the error syndrome of the Hamming code.

Let $a$ be the largest integer such that $2^a \le \frac{d}{2}$. We note that $2^a > \frac{d}{4}$. Partition the input vector $v \in \{0, 1\}^d$ to two vectors, each of size $2^a$ bits, ignoring leftover bits. By flipping a single bit in each of the vectors of size $2^a$ we can transmit $a$ bits of information, so in total the number of message options is
$$m = 2^a \cdot 2^a > \parens{\frac{d}{4}}^2$$
as required.
\end{proof}

\subsection{Cycle-breaking set of permutations}
\label{sec:breaking}
We now construct a set of permutations $T \subset S_n$ such that any permutation of $n$ elements can be composed with a permutation in $T$ to obtain a permutation with no cycles larger than $k = n/u$.

We note that we use $S_n$ here for simplicity. In practice we apply the results of this section to $S_{n-r}$ as described in \Cref{sec:scheme}.

\subsubsection{Ramanujan graphs}

We note that this is the most involved part of our work, using concepts such as expander graphs, Ramanujan graphs and the expander mixing lemma. We summarize our usage of this tool in \Cref{cor:ramanujan} below.

We use the famous construction of Ramanujan graphs \cite{ramanujan}: Given two prime numbers $p$ and $q$ such that $p \equiv q \equiv 1 \pmod{4}$, there is a $(p+1)$-regular Ramanujan graph of $n=q(q^2-1)/2$ or $n=q(q^2-1)$ vertices depending on whether or not $p$ is a quadratic residue modulo $q$. Given $p$ and $q$ the graph is constructed efficiently using a simple formula.

Since this graph is Ramanujan, the expander mixing lemma implies that for any two disjoint sets of vertices $V_1$ and $V_2$ the number of edges between the two sets $e(V_1, V_2)$ satisfies:
$$\left| e(V_1, V_2) - \frac{(p + 1) |V_1| |V_2|}{n} \right| \le 2 \sqrt{p |V_1| |V_2|}$$
In other words, the number of such edges is approximately the number obtained if every edge exists independently with probability $\frac{p + 1}{n}$.

If both $V_1$ and $V_2$ to contain at least $x \cdot n$ vertices each ($x < 1$) the following bound can be derived:
$$e(V_1, V_2) \ge (p + 1)x^2n - 2 \sqrt{p} x n$$

The total number of edges in the graph is $\frac{1}{2}(p+1)n$. We can normalize $e(V_1, V_2)$ by this number to obtain:
$$\frac{2}{(p+1)n}e(V_1, V_2) \ge 2x^2 - \frac{4x\sqrt{p}}{p+1} \ge 2x^2 - \frac{4x}{\sqrt{p}}$$

If $p \ge 16 / x^2$  we have $\frac{4x}{\sqrt{p}} \le x^2$ and the above bound implies the simpler bound:
$$\frac{2}{(p+1)n}e(V_1, V_2) \ge x^2$$

We summarize the only takeaway of this section that we will be using:

\begin{corollary}
\label{cor:ramanujan}
Given two prime numbers $p$ and $q$ such that $p \equiv q \equiv 1 \pmod{4}$, there is an efficiently constructed $(p+1)$-regular graph of $n=q(q^2-1)/2$ or $n=q(q^2-1)$ vertices depending on whether or not $p$ is a quadratic residue modulo $q$. If $p \ge 16 / x^2$, any two disjoint subsets of vertices of this graph of size at least $x \cdot n$ have an edge between them. Moreover, the number of edges between the subsets is at least $x^2$ of the total number of edges in the graph.
\end{corollary}

\subsubsection{Cycle-breaking set of transpositions}
\label{sec:edges}
Our goal is to construct a cycle-breaking set of permutations, such that the cycles of any given permutation can be broken down to be smaller than $k=n/u$ by composing with a permutation from the set. In this section we move one step towards this goal by proving the following lemma.

\begin{lemma}
\label{lemma:edges}
Given $n$ and $u \ge 1$, one can efficiently construct a set of $s = O(u^2 n)$ transpositions on $n$ elements such that the cycles of any given permutation $\pi \in S_n$ can be broken down to be smaller than $k=n/u$ by composing $\pi$ with no more than $2u$ of the predetermined transpositions.

Moreover, there are many ways to choose the transpositions. In fact, for any permutation $\pi \in S_n$ there are $\ell \le 2u$ subsets $W_1, ..., W_\ell$ of the predetermined transpositions such that any choice of $\ell$ transpositions, one from each $W_i$ breaks the cycles of $\pi$ to be smaller than $k=n/u$. Each such subset satisfies $|W_i|\ge \frac{s}{16u^2}$.
\end{lemma}

\begin{proof}
The set of transpositions will be determined by the set of edges of a Ramanujan graph with $n$ vertices that correspond to the numbers $1, ..., n$ that $S_n$ acts on. Every edge specifies a transposition by describing the pair of indices that are swapped.

We construct the Ramanujan graph by choosing $p_0 \ge 256 u^2$, so \Cref{cor:ramanujan} guarantees any two sets of vertices of size $\frac{n}{4u}$ are connected by at least $\frac{1}{16u^2}$ of the edges in the graph. The number of transpositions constructed from the edges of this graph is $\frac{p_0 + 1}{2} n = O(u^2n)$ as required by the lemma.

We note that the requirements of \Cref{cor:ramanujan} do not allow arbitrary choices of the degree and size of the graph. We will need to increase $p_0$ and $q$ slightly in order to satisfy the requirements. There are number-theoretic results that prove that this does not change the conclusion of \Cref{lemma:edges}, see for example \cite{expanders}. For brevity this issue is ignored in this analysis.

Given a permutation $\pi \in S_n$, its cycles can be broken using the set of transpositions in the following way. We handle each cycle of $\pi$ separately. We partition the cycle into consecutive sets of vertices of size not larger than $\frac{n}{4u}$. See \Cref{fig:cycle} for an illustration. If the number of vertices in the cycle is not divisible by $\frac{n}{4u}$ the last set would be smaller.

\begin{figure}
    \centering
\begin{tikzpicture}
\draw[dotted,thick] circle [radius=50pt];
\draw[blue] circle [radius=65pt];
\foreach \angle in {30, 90, 150, 210, 270, 330} {
    \draw[blue] (\angle:60pt) -- +(\angle:10pt);
}
\node[fill=white] at (60:65pt) {\textcolor{blue}{$V_1$}};
\node[fill=white] at (0:65pt) {\textcolor{blue}{$V_2$}};
\node[fill=white] at (300:65pt) {\textcolor{blue}{$V_3$}};
\node[fill=white] at (240:65pt) {\textcolor{blue}{$V_4$}};
\node[fill=white] at (180:65pt) {\textcolor{blue}{$V_5$}};
\node[fill=white] at (120:65pt) {\textcolor{blue}{$V_6$}};
\draw[dotted,red,thick,<->]
  (35.355pt, 35.355pt)
  arc[start angle=-45,end angle=-135,radius=50pt];
\draw[dotted,red,thick,<->]
  (35.355pt, -35.355pt)
  arc[start angle=45,end angle=135,radius=50pt];
\draw[dotted,red,thick,<->]
  (50pt, 0pt) -- (-50pt, 0pt);
\end{tikzpicture}

\caption{Cycle breaking example. The circle of  black dots corresponds to a cycle of a permutation $\pi$, such that applying $\pi$ on each vertex yields the next vertex clockwise. We divide the cycle into six sets of vertices $V_1, ..., V_6$. Three edges are chosen between the sets and are marked in red. By composing $\pi$ with the transpositions that correspond to the chosen edges, the cycle is broken down into four small cycles.}
\label{fig:cycle}
\end{figure}
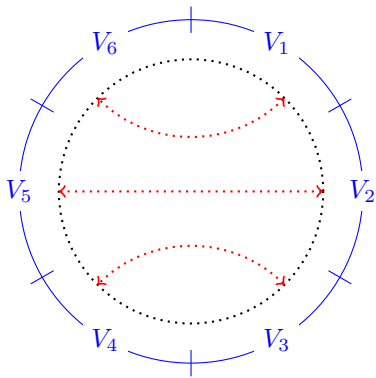

Since there is an edge between every two sets, we can pick an edge specifically between pairs of sets that are a reflection of each other along an arbitrary fixed axis. If a cycle is partitioned into sets of vertices $V_1, ..., V_t$, numbered in the direction the permutation acts, then we pick an edge between $V_1$ and $V_t$, another edge between $V_2$ and $V_{t-1}$ and so on. These edges are marked in red in \Cref{fig:cycle}. If the number of sets is odd then one set is left unconnected. We choose this set to be the middle set, $V_{(t+1)/2}$ in the above indexing. Furthermore, if the cycle is not divisible by $\frac{n}{4u}$ we require the smaller set with the division reminder to be left unconnected, so it must be one of the middle sets.

Every transposition between elements of the same cycle breaks it into two smaller cycles. The suggested transpositions result in smaller cycles, each cycle containing elements of up to four sets, so it is not larger than $\frac{n}{u}$ as required. The number of transpositions made is at most half the number of sets, so out of the $\frac{1}{2}(p_0+1)n$ edges of the graph, up to $2u$ are chosen.
\end{proof}

We note that the transpositions that break the cycles of a specific permutation in \Cref{lemma:edges} are disjoint, in the sense that no two of them swap the same number. This means the chosen transpositions commute, so there is no need to specify the order they are composed.

\subsubsection{Subsets of transpositions}
\Cref{lemma:edges} constructs a set of $s = \frac{1}{2}(p_0+1)n$ transpositions such that the cycles of any permutation can be broken by applying up to $2u$ of the predetermined transpositions. An important feature of this construction is that there are many ways to choose each of the transpositions. In fact, \Cref{cor:ramanujan} shows that at least $\frac{s}{16u^2}$ of the predetermined transpositions can be used interchangeably for each of the $2u$ required transpositions.

Our goal is to construct a collection of subsets $T_1, ..., T_m$ of the $s$ transpositions, each subset $T_i$ containing $2u$ transpositions. Given any list of $2u$ disjoint subsets of transpositions $W_1, ..., W_{2u}$, each $W_i$ of size $\ge \frac{s}{16u^2}$, we would like at least one predetermined subset $T_i$ to intersect with every one of $W_1, ..., W_{2u}$, so $T_i$ consists of one transposition from each set $W_j$, and no other transpositions. Having achieved this the spy can choose the sets $W_i$ according to \Cref{lemma:edges} and obtain one of the predetermined subsets of transpositions $T_i$ that breaks all cycles to be smaller than $k = n/u$.

\begin{remark*}
We note that the problem of constructing such a collection of subsets is solved by a generalization of the expander mixing lemma for hypergraphs, which is a special case of expander complexes. This result is proven in \cite[Theorem 1.4]{mixing_complex}. A construction of Ramanujan complexes (which are expander complexes with optimal spectral gap) is given in \cite{ramanujan_complex}. A survey that includes both the construction and the lemma is given in \cite{survey_complex}. Here we avoid using these more advanced results in favor of a simpler construction using Ramanujan graphs.
\end{remark*}

In order to construct the sets $T_1, ..., T_m$ we recursively construct Ramanujan graphs. On the first iteration we consider the given $s$ transpositions as \textit{vertices} of a Ramanujan graph. The edges of this graph provide a list of pairs of transpositions, such that any two subsets of the original $s$ transpositions that are large enough have an edge between them. In the second iteration we consider the \textit{edges} of the Ramanujan graph constructed in the first iteration as \textit{vertices} of a new Ramanujan graph. Each edge of the Ramanujan graph corresponds to a pair of pairs of transpositions. On the $t$-th iteration each edge corresponds to a set of $2^t$ transpositions. The required sets of transpositions $T_1, ..., T_m$ are chosen according to the edges of the Ramanujan graph of the final iteration.

We now provide the parameters for the graph constructed in each iteration. On the first iteration we construct a Ramanujan graph with $s$ vertices and $p_1 > 16 \parens{16u^2}^2 = 4096 u^4$. \Cref{cor:ramanujan} ensures that at least $\frac{1}{\parens{16u^2}^2}$ of the edges are between every two sets of transpositions each containing at least $\frac{1}{16u^2}$ of the transpositions.

On each iteration we consider the edges of the previous iteration as vertices of a new Ramanujan graph with $p > 16/x^2$ where $x$ is the portion of edges guaranteed by the previous iteration. This way at least $x^2$ of the edges of the new iteration are between any two sets containing at least a portion $x$ of the edges of the previous iteration.

At the $t$-th iteration (starting at $t=0$) we have a collection of sets, each containing $2^t$ of the original transpositions. The proportion of sets guaranteed at the $t$-th iteration is $\frac{1}{\parens{16u^2}^{2^t}}$. We will need to pick a prime $p_t > 16 \parens{16u^2}^{2^{t+1}}$ increasing the number of edges by $p_t + 1$.

We repeat this process for $\tau$ steps, and require the number of transpositions in the final sets $2^\tau$ to be at least the number of required transpositions $2u$, so $2^\tau$ is bounded by $4u$. The number of sets constructed this way is approximately:
$$s \cdot \prod_{t=0}^{\tau-1} p_t = s \cdot \prod_{t=0}^{\tau-1} 16 \parens{16u^2}^{2^{t+1}} = s \cdot 16^\tau \parens{16u^2}^{2 + 4 + ... + 2^\tau}$$

Recalling that $s = \frac{1}{2}(p_0+1)n$ for $p_0 \ge 256 u^2$, the total number of sets is approximately:
$$128n u^2 \cdot 16^\tau \parens{16u^2}^{2 + 4 + ... + 2^\tau} = 8 n \cdot 16^\tau \parens{16u^2}^{1 + 2 + 4 + ... + 2^\tau}$$
which is bounded by
$$8 n \cdot 16^\tau \parens{16u^2}^{2 \cdot 2^\tau} \le 8 n \cdot (4u)^4 \parens{16u^2}^{8u} = 8 n \cdot \parens{4u} ^ {16u + 4}$$

In conclusion, we have constructed a set of no more than $8 n \cdot \parens{4u} ^ {16u + 4}$ permutations (each one is a set of transpositions) such that the cycles of any permutation can be broken down into cycles not larger than $n/u$ by composing with one of the predetermined permutations.

\begin{remark*}
In order for the construction to work we may need to apply a number of transpositions that is not a power of two. The easiest way to allow this is to add dummy elements to the set of transpositions constructed in \Cref{sec:edges}. This slightly increases the number of required sets, which is also increased in order to meet the requirements of \Cref{cor:ramanujan}. These considerations are ignored for brevity, as they do not change the bottom line. 
\end{remark*}

\subsection{Connecting the components}
In this section we connect the components to prove the theorem given in the introduction:
\deterministic*
\begin{proof}
The spy and prisoners use the scheme described in \Cref{sec:scheme}. All prisoners open the first $r$ drawers. Using \Cref{theorem:swap2flips} and \Cref{theorem:flips2message}, the spy communicates to all prisoners a message out of approximately $\parens{\frac{r}{12}}^2$ options by making a single swap between the first $r$ drawers. The transmitted message corresponds to one of the predetermined permutations on the $n-r$ numbers not in the first $r$ drawers. The set of permutations is constructed in \Cref{sec:breaking}. There are up to $8 (n - r) \cdot \parens{4u} ^ {16u + 4}$ permutations in the set, each permutation consisting of up to $2u$ transpositions. The spy can always find a predetermined permutation such that composing it with the permutation defined by the remaining $n-r$ drawers yields a permutation with no cycles larger than $\frac{n-r}{u}$.

Following this strategy the total number of drawers opened by a prisoner in the worst case is:
$$r + \frac{n-r}{u}$$

For the construction to work the number of options encoded in the first $r$ drawers must be at least as large as the number of predetermined permutations:
\begin{equation}
\label{eq:enough_options}
\parens{\frac{r}{12}}^2 \ge 8 (n - r) \cdot \parens{4u} ^ {16u + 4}
\end{equation}

By choosing for example $16u + 4 = (1 - \epsilon) \frac{\log n}{\log \log n}$ for any $\epsilon > 0$ we obtain:
$$\parens{4u} ^ {16u + 4} < n ^ {1-\epsilon}$$
so we can pick $r = O(n^{1 - \epsilon/2})$ that satisfies \Cref{eq:enough_options} to obtain:
$$r + \frac{n-r}{u} = O\parens{\frac{n \log \log n}{\log n}}$$
as required.

As for the time complexity of running this strategy, note that the number of candidate permutations in the strategy is bounded by $8n^{2 - \epsilon}$. Each permutation consists of up to $2u = O(\log n)$ transpositions, so all permutations can be written down and numbered by the spy and prisoners in time $O(n^2)$. The spy can find the cycle structure of the last $n-r$ drawers in time $O(n)$. Using the cycle structure every candidate permutation can be checked in time polynomial in the number of transpositions, so the spy can find the cycle-breaking permutation in time $O(n^2)$. The spy can then encode the index of the chosen permutation in time $O(r)$ as explained in \Cref{sec:encoding}. So in total both spy and prisoners can run the strategy in time $O(n^2)$ each.

We note that the participants need to find prime numbers that satisfy the conditions of \Cref{cor:ramanujan}. It can be checked that the largest needed prime is of size $O(n)$. The time of finding all primes up to this bound is much less than $O(n^2)$, so this part is negligible.
\end{proof}

\section*{Acknowledgement}
We thank Dan Karliner for posing the question of whether opening half of the drawers is necessary to guarantee the prisoners' freedom. We also thank Noam Kimmel, Dean Hirsch, Amit Atsmon, Ido Kessler and Tom Kalvari for helpful insights and suggestions.

\bibliography{main}

\end{document}